\documentclass[draft]{amsart}
\usepackage{amsmath}
\usepackage{amssymb}

\newtheorem{thm}{Theorem}
\newtheorem{lem}[thm]{Lemma}
\newtheorem{cor}[thm]{Corollary}
\newtheorem{prop}[thm]{Proposition}

\theoremstyle{remark}
\newtheorem{rmk}[thm]{Remark}

\newcommand{\cv}{\mathbb{C}}

\newcommand{\aut}{\textup{Aut}}

\numberwithin{thm}{section}
\numberwithin{equation}{section}

\begin{document}

\title{The Burns-Krantz rigidity with an interior fixed point}

\author{Feng Rong}

\address{Department of Mathematics, School of Mathematical Sciences, Shanghai Jiao Tong University, 800 Dong Chuan Road, Shanghai, 200240, P.R. China}
\email{frong@sjtu.edu.cn}



\thanks{The author is partially supported by the National Natural Science Foundation of China (Grant No. 12271350).}

\begin{abstract}
We prove a Burns-Krantz type boundary rigidity near strongly pseudoconvex points for holomorphic self-maps with an interior fixed point. This confirms a conjecture of Huang.
\end{abstract}

\maketitle

\section{Introduction}

In \cite{BK}, Burns and Krantz established the following fundamental rigidity result.

\begin{thm}\cite{BK}
Let $\Omega$ be a smoothly bounded strongly pseudoconvex domain in $\cv^n$, $n\ge 1$, $p\in \partial \Omega$, and $F$ be a holomorphic self-map of $\Omega$. If $F(z)=z+o(|z-p|^3)$ as $z\rightarrow p$, then $F(z)\equiv z$.
\end{thm}

In \cite{H:CJM95}, Huang proved a ``localized" version of the Burns-Krantz rigidity (\cite[Theorem 2.5]{H:CJM95}), assuming only strong pseudoconvexity near the boundary point $p$. Furthermore, the following two Burns-Krantz type rigidity results, with an interior fixed point, were proven.

\begin{thm}\cite[Corollary 1.5]{H:CJM95}\label{T:HD}
Let $D$ be a smoothly bounded domain in $\cv$, $p\in \partial D$, and $F$ be a holomorphic self-map of $D$. If $F(z_0)=z_0$ for some $z_0\in D$ and $F(z)=z+o(|z-p|)$ as $z\rightarrow p$, then $F(z)\equiv z$.
\end{thm}

\begin{thm}\cite[Corollary 2.7]{H:CJM95}\label{T:HO}
Let $\Omega$ be a smoothly bounded strongly convex domain in $\cv^n$, $n\ge 2$, $p\in \partial\Omega$, and $F$ be a holomorphic self-map of $\Omega$. If $F(z_0)=z_0$ for some $z_0\in \Omega$ and $F(z)=z+o(|z-p|^2)$ as $z\rightarrow p$, then $F(z)\equiv z$.
\end{thm}

Moreover, it was conjectured in \cite{H:CJM95} that a similar result should hold for bounded strongly pseudoconvex domains in $\cv^n$, $n\ge 2$. With the extra assumption that $\Omega$ is simply-connected, this was proven by Huang in \cite{H:IJM94} and \cite{H:SNS94} (cf. \cite{FR}). The following result is a combination of \cite[Corollary 3]{H:IJM94} and \cite[Theorem 5]{H:SNS94}.

\begin{thm}
Let $\Omega\subset \cv^n$, $n\ge 2$, be either a bounded simply-connected taut domain with Stein neighborhood basis or a bounded simply-connected pseudoconvex domain with $C^\infty$-smooth boundary. Let $p\in \partial \Omega$ be a $C^3$-smooth strongly pseudoconvex point. Let $F$ be a holomorphic self-map of $\Omega$ such that $F(q)=q$ for some $q\in \Omega$ and
$$F(z)=z+o(\|z-p\|^2),\ \ \ z\rightarrow p.$$
Then $F(z)\equiv z$.
\end{thm}

The aim of this note is to show that Huang's conjecture is true. Indeed, it would be a corollary of our main result below.

\begin{thm}\label{T:main}
Let $\Omega\subset \cv^n$, $n\ge 2$, be either a bounded taut domain with Stein neighborhood basis or a bounded pseudoconvex domain with $C^\infty$-smooth boundary. Let $p\in \partial \Omega$ be a $C^3$-smooth strongly pseudoconvex point. Let $F$ be a holomorphic self-map of $\Omega$ such that $F(z)=z+o(|z-p|^k)$ as $z\rightarrow p$. If $F(z)\not\equiv z$, then $k\le 2$ and\\
1) If $k=2$, then $\{F^m\}$ converges normally to $p$ and $F$ can not be an automorphism of $\Omega$.\\
2) If $k=1$, then either $\{F^m\}$ converges normally to $p$ or $F$ fixes a holomorphic retract with positive dimension. And $F$ can not be an automorphism of $\Omega$ unless $\Omega$ is biholomorphic to the unit ball.
\end{thm}

As an immediate corollary, we have the following.

\begin{cor}\label{C:main}
Let $\Omega$ and $p\in \partial \Omega$ be as in Theorem \ref{T:main}. If $F$ is a holomorphic self-map of $\Omega$ such that $F(z_0)=z_0$ for some $z_0\in \Omega$ and $F(z)=z+o(|z-p|^2)$ as $z\rightarrow p$, then $F(z)\equiv z$.
\end{cor}

Since a bounded strongly pseudoconvex domain with $C^2$-smooth boundary is taut and has Stein neighborhood basis, Corollary \ref{C:main} confirms Huang's conjecture.

\begin{rmk}
With the extra assumption that $\Omega$ is simply-connected, Theorem \ref{T:main} was given in \cite{H:SNS94}. Our main contributions are Lemmas \ref{L:M} and \ref{L:phi}.
\end{rmk}

\section{Proof of Theorem \ref{T:main}}

Let us first recall the following fundamental result from the iteration theory on taut manifolds (see e.g. \cite{B:Stein,A:book}).

\begin{thm}\label{T:taut}
Let $X$ be a taut manifold and $F$ be a holomorphic self-map of $X$. Assume that $\{F^m\}$ is not compactly divergent. Then there is a holomorphic retraction $\rho$ of $X$ onto a submanifold $M$ such that every limit point $h$ of $\{F^m\}$ is of the form $h=\phi\circ \rho$, for a suitable $\phi\in \aut(M)$. Furthermore, $\rho$ is a limit point of $\{F^m\}$, and $f:=F|_M\in \aut(M)$.
\end{thm}

Recall that a holomorphic self-map $\rho$ of $X$ is a \textit{holomorphic retraction} if $\rho^2=\rho$. The submanifold $M=\rho(X)$ is called a \textit{holomorphic retract} of $X$.

Next, let us quote some results from Huang's work \cite{H:IJM94,H:SNS94,H:CJM95}.

\begin{lem}\cite[Lemma 6]{H:IJM94}\label{L:auto}
Let $\Omega$ and $p\in \partial \Omega$ be as in Theorem \ref{T:main}. Let $F$ be an automorphism of $\Omega$ such that $F(z)=z+o(|z-p|^k)$ as $z\rightarrow p$. If either $k=2$ or $k=1$ and $\Omega$ is not biholomorphic to the unit ball, then $F(z)\equiv z$.
\end{lem}

\begin{lem}\cite[Lemma 7]{H:IJM94}\label{L:V}
Let $\Omega$ and $p\in \partial \Omega$ be as in Theorem \ref{T:main}. Let $F$ be a holomorphic self-map of $\Omega$ such that $F(z)=z+o(|z-p|)$ as $z\rightarrow p$. Then for any neighborhood $V$ of $p$, there exists a point $z\in V\cap \Omega$ such that $f^m(z)\in V$ for any $m\ge 1$.
\end{lem}

\begin{prop}\cite[Proposition 1]{H:SNS94}\label{P:cg}
Let $\Omega$ and $p\in \partial \Omega$ be as in Theorem \ref{T:main}. Suppose that $M\subset \Omega$ is a holomorphic retract with complex dimension greater than 1 and that $p\in \partial M$. Then, for any neighborhood $U$ of $p$, there is a complex geodesic $\varphi$ of $\Omega$ with $\varphi(\Delta)\subset U\cap M$, $\varphi(1)=p$, and $\varphi\in C^1(\overline{\Delta})$.
\end{prop}

Recall that a holomorphic map $\varphi:\Delta\rightarrow \Omega$ is called a \textit{complex geodesic} of $\Omega$ (see e.g. \cite{A:book}), if it is an isometry between $k_\Delta$ and $k_\Omega$, i.e. $k_\Omega(\varphi(\zeta_1),\varphi(\zeta_2))= k_\Delta(\zeta_1,\zeta_2)$ for all $\zeta_1,\zeta_2\in \Delta$. Here, $k_\Omega(z_1,z_2)$ denotes the \textit{Kobayashi distance} of $z_1,z_2\in \Omega$ and $\Delta$ is the unit disk.

\begin{thm}\cite[Theorem 2.2]{H:CJM95}\label{T:Huang}
Let $\Omega$ be a bounded domain in $\cv^n$, $n\ge 2$. Let $\varphi:\Delta\rightarrow \Omega$ be a proper holomorphic map which is Lipschitz continuous near $\varphi(1)=p\in \partial\Omega$. Suppose that there exists a defining function $\rho$ of $\Omega$ which is $C^1$-smooth, with $d\rho\neq 0$, near $p$ such that $\rho\circ \varphi$ is subharmonic. If $F$ is a holomorphic self-map of $\Omega$ fixing $\varphi(\Delta)$ with $F(z)=z+o(|z-p|^2)$ as $z\rightarrow p$, then $F(z)\equiv z$.
\end{thm}

We also need the following lemma.

\begin{lem}\cite[Lemma 2.2]{BZ:jets}\label{L:jets}
Let $\Omega$ be a domain in $\cv^{n+1}$ having $p$ as a smooth boundary point and $f:\Omega\rightarrow \cv^m$ be holomorphic. Suppose that the non-tangential differential $df_p$ of $f$ at $p$ exists. Then $df_{z_k}\rightarrow df_p$ for any sequence $\{z_k\}$ in $\Omega$ converging non-tangentially to $p$.
\end{lem}

Now, let $\Omega$, $p\in \partial \Omega$ and $F:\Omega\rightarrow \Omega$ be as in Theorem \ref{T:main}, with $F(z)\not\equiv z$.

First, suppose that $\{F^m\}$ is compactly divergent. Then we have $k\le 2$ by Huang's localized Burns-Krantz rigidity (\cite[Theorem 2.5]{H:CJM95}). Moreover, by Lemma \ref{L:auto}, $F$ can not be an automorphism unless $k=1$ and $\Omega$ is biholomorphic to the unit ball.

Next, suppose that $\{F^m\}$ is not compactly divergent. Then, by Theorem \ref{T:taut}, there is a holomorphic retract $M$ of $\Omega$. Set $l=\textup{dim}_\cv M$. By Lemma \ref{L:V}, $l>0$ and $p\in \partial M$. And by Lemma \ref{L:auto}, $l<n$. Thus, we can assume that $1\le l\le n-1$.

\begin{lem}\label{L:M}
For $1\le l\le n-1$, set $f=F|_M$. Then $f(z)\equiv z$.
\end{lem}
\begin{proof}
When $l=1$, by the classification theorem \cite[Theorem 5.2]{M}, either $f$ is an irrational rotation on $M$ (with $M$ isomorphic either to a disk, to a punctured disk or to an annulus) or $f$ is of finite order, i.e. there exists $m\ge 1$ such that $f^m\equiv id$, and every point of $M$ is periodic. From the work of \cite{FR} (cf. \cite[Theorem 1]{H:IJM94}), we know that $M$, near $p$, is a Lipschitz graph over an open set in the complex plane through $p$ containing the complex normal direction to $\partial \Omega$. Thus, $dF_z\rightarrow id$ as $z\rightarrow p$ in $M$ by Lemma \ref{L:jets}, since $F(z)=z+o(|z-p|)$ as $z\rightarrow p$. Hence, we must have $f(z)\equiv z$.

When $2\le l\le n-1$, the lemma is essentially \cite[Lemma 5]{H:SNS94}. (The domain is assumed to be simply-connected in \cite[Lemma 5]{H:SNS94}, although this assumption is not needed in its proof.)
\end{proof}

\begin{lem}\label{L:phi}
When $l=1$, let $\varphi:\Delta\rightarrow M\subset \Omega$ be either the Riemann mapping (if $M$ is simply-connected) or the universal covering map (if $M$ is not simply-connected), with $\varphi(\tau_j)\rightarrow p$ for a sequence $\tau_j\rightarrow 1$. When $2\le l\le n-1$, let $\varphi:\Delta\rightarrow \Omega$ be a complex geodesic of $\Omega$ contained in $M$ with $\varphi(\tau_j)\rightarrow p$ for a sequence $\tau_j\rightarrow 1$. Then, $\varphi$ is Lipschitz continuous near $1$, with $\varphi(1)=p$.
\end{lem}
\begin{proof}
When $l=1$ and $M$ is simply-connected, the lemma is just \cite[Lemma 8]{FR}.

When $l=1$ and $M$ is not simply-connected, the proof is essentially the same as that of \cite[Lemma 8]{FR}. Indeed, although \cite[Lemma 8]{FR} deals with the biholomorphic Riemann mapping $\varphi$, its proof only needs the local inverse of $\varphi$, which is also valid for the universal covering map.

When $2\le l\le n-1$, the lemma follows from Proposition \ref{P:cg}.
\end{proof}

By Lemmas \ref{L:M} and \ref{L:phi}, Theorem \ref{T:Huang} can be applied to $F$ and shows that $k=1$ since $F(z)\not\equiv z$. This completes the proof of Theorem \ref{T:main}.

\begin{rmk}
In both Theorem \ref{T:main} and Corollary \ref{C:main}, we only need to assume the boundary behavior of $F(z)$ as $z\rightarrow p$ non-tangentially. For a more detailed discussion in this direction, see \cite{R:Non}.
\end{rmk}

\end{document}